\documentclass[12pt,a4paper]{amsart}

\usepackage{amsmath,enumerate}
\usepackage{amsfonts}
\usepackage{amsthm}
\usepackage{amssymb}
\usepackage[english]{babel}
\usepackage{graphicx}
\usepackage[all]{xy}

\usepackage[active]{srcltx}
\usepackage[parfill]{parskip}

\newtheorem{theorem}{Theorem}
\newtheorem{lemma}[theorem]{Lemma}

\newtheorem{corollary}[theorem]{Corollary}
\newtheorem{proposition}[theorem]{Proposition}
\newtheorem{remark}[theorem]{Remark}

\newcommand{\Hom}{{\mathrm{Hom}}}
\newcommand{\Ext}{{\mathrm{Ext}}}

\newcommand{\eins}{\leavevmode\hbox{\small1\kern-3.8pt\normalsize1}}

\newcommand{\res}{{\rm Res} }

\newcommand{\ind}{{\rm Ind} }

\newcommand{\tto}{\twoheadrightarrow}

\newcommand{\mC}{\mathbb{C}}
\newcommand{\mN}{\mathbb{N}}

\newcommand{\mZ}{\mathbb{Z}}

\newcommand{\mm}{\mathfrak{m}}

\newcommand{\fa}{{\mathfrak a}}
\newcommand{\fb}{{\mathfrak b}}

\newcommand{\fg}{{\mathfrak g}}
\newcommand{\fh}{{\mathfrak h}}

\newcommand{\fn}{{\mathfrak n}}

\newcommand{\fm}{{\mathfrak m}}

\newcommand{\cD}{\mathcal{D}}

\newcommand{\cH}{\mathcal{H}}

\newcommand{\cF}{\mathcal{F}}

\newcommand{\cC}{\mathcal{C}}

\newcommand{\cW}{\mathcal{W}}
\newcommand{\cO}{\mathcal{O}}

\newcommand{\cA}{\mathcal{A}}

\newcommand{\cB}{\mathcal{B}}
\newcommand{\cZ}{\mathcal{Z}}
\newcommand{\cY}{\mathcal{Y}}
\newcommand{\cX}{\mathcal{X}}
\newcommand{\cI}{\mathcal{I}}

\newcommand{\pd}{{\rm pd}}

\begin{document}

\title[Homological properties of $\cO$. III]{Some homological properties\\ of category $\cO$. III}

\author{Kevin Coulembier and Volodymyr Mazorchuk}
\date{}

\begin{abstract}
We prove that thick category $\mathcal{O}$ associated to a semi-simple complex
finite dimensional Lie algebra is extension full in the category of all modules.
We also prove the weak Alexandru conjecture both for 
regular blocks of thick category $\mathcal{O}$
and the associated categories of Harish-Chandra bimodules, but disprove it for singular blocks. 
\end{abstract}

\maketitle



\section{Introduction}

The main motivation for the present paper is 
the so-called {\em (weak) Alexandru con\-jec\-tu\-re} 
as stated in the two arXiv preprints \cite{Ga1,Ga2} by  Pierre-Yves Gaillard.
Slight variations of these conjectures were studied in the PhD Thesis \cite{Fu1} of
Alain Fuser and further popularized in the series \cite{Fu2,Fu3,Fu4,Fu5} of preprints and 
manuscripts by the same author.\footnote{These manuscripts were previously available online and can be obtained from the authors of the current paper on request.} The conjectures concern
certain homological properties of various categories of Harish-Chandra modules over real and complex Lie algebras
modelled on the classical properties of the BGG category~$\mathcal{O}$ from \cite{BGG}.

Given an abelian category $\cA$, Yoneda defined the extension groups $\Ext^d_{\cA}(M,N)$ for any $M,N\in\cA$ and 
$d\geq 0$ using equivalence classes of exact sequences of length $d+2$. For any abelian subcategory
$\cB$ of $\cA$ with exact inclusion, 
the definition gives rise to a canonical map $\Ext^d_{\cB}(M,N)\to \Ext^d_{\cA}(M,N)$ which
is neither injective not surjective in general.  We say that $\cB$ is {\em extension full} in $\cA$
if these canonical maps are isomorphisms for any $M,N$ and $d$. Weak Alexandru conjecture could be roughly simplified
to the conjecture that certain subcategories of categories of Harish-Chandra modules are extension full.

The property of being extension full in this context is motivated by a famous theorem of Cline, Parshall and Scott
from \cite{CPS1}, which asserts that the Serre subcategory associated with a coideal of the partially
ordered set indexing simple objects of some highest weight category $\mathcal{C}$ is extension full in $\mathcal{C}$.
All definitions are designed so that this result of \cite{CPS1}, combined with well-known consequences of the
Kazhdan-Lusztig conjecture (see \cite{MR2428237}), automatically implies that weak Alexandru conjecture is 
true for the principal block $\mathcal{O}_0$, we prove this in detail in Theorem~\ref{WAC} below. We also prove 
weak Alexandru conjecture for thick category $\cO_0$, but disprove it for a singular block in category $\cO$.
To the best of our knowledge, the general case of (weak) Alexandru conjectures is still open. The result on
singular blocks in category $\cO$ shows that the properties required by weak Alexandru conjecture
are less natural than and not equivalent to the extension fullness result in~\cite{CPS1}.

Extension fullness, the key notion behind weak Alexandru conjectures, seems to be an interesting and 
non-trivial property. The aim of this paper is to investigate extension fullness for various pairs of
categories of modules over complex semi-simple Lie algebras and basic classical Lie superalgebras, which appear 
in the context of Alexandru conjectures. For this purpose we derive several criteria for extension fullness for two abelian categories in a general abstract setting, which we then apply to categories of Lie algebra modules. Here is a short list of our main results:
\begin{itemize}
\item Category $\mathcal{O}$ is extension full in the category of weight modules.
\item Thick category $\mathcal{O}$ is extension full in the category of all modules.
\item The category of generalized weight modules is extension full in the category of all modules.
\item Confirmation of weak Alexandru conjecture for the principal block of thick category $\mathcal{O}$ 
and the associated category of Harish-Chandra bimodules.
\item Disproof of weak Alexandru conjecture for a singular block in category $\cO$.
\item Computation of projective dimension, inside the thick category~$\mathcal{O}$, of structural 
modules from the usual category $\mathcal{O}$.
\item For any module $M$ which has finite projective dimension in any of the finite thick versions $\cO^I$ of $\cO$ as defined in in Subsection~\ref{secPrel.5}, we have $\pd_{\cO^\infty}M=\pd_{\cO^I}M+\dim\fh$.
\end{itemize}

The first of these results was stated (without proof) in~\cite{De}, the other results are new. A remaining open question is whether the Alexandru conjecture holds for singular blocks in thick category $\cO$. The example in Section \ref{secsing.2}, together with Lemma~\ref{initseg}(ii), seems to provide a good candidate to disprove this.

The paper is organized as follows. Section~\ref{secPrel} provides necessary background from homological algebra.
Section~\ref{secExtFull} gives several effective criteria to check extension fullness for abelian categories in an
abstract situation. In Section~\ref{secO} we prove that category $\mathcal{O}$ is extension full in the 
category of weight modules and that thick category $\cO$ is extension full in the category of generalized weight modules. In Section~\ref{secOinf} we show that thick category $\mathcal{O}$ is extension 
full in the category of all modules and even reduce computation of projective dimension for objects in the
thick category $\mathcal{O}$ to computation of projective dimension in the usual category $\mathcal{O}$. In Section~\ref{secsing} we focus on some basic homological properties in singular blocks of category $\cO$.
Section~\ref{secwac} proves weak Alexandru conjecture for regular blocks of (thick) category $\mathcal{O}$
and disproves it for some singular blocks of (thick)  $\mathcal{O}$ based on an examples 
described in Section~\ref{secsing}.
Finally, in Section~\ref{sechc} we extend our results to the category of Harish-Chandra bimodules.

The concept which we call an `extension full subcategory' appeared in other very recent work with different terminology. In \cite{Ps} this concept is referred to as a `homological embedding' and in \cite{He} as an `entirely extension closed subcategory'. 

Despite the fact that we do use some results from the first two papers \cite{MR2366357,MR2602033} 
in the series, the present paper is rather a complement to than a continuation of \cite{MR2366357,MR2602033}.

\section{Preliminaries}
\label{secPrel}

We denote by $\mathbb{N}$ the set of all non-negative integers.
All subcategories are assumed to be full.
We abbreviate $\otimes_{\mathbb{C}}$ by $\otimes$.

\subsection{Extensions}
\label{secPrel.1}

We start with recalling the classical approach of Yoneda, see \cite{Buchsbaum} or \cite[Vista~3.4.6]{Weibel},
to the definition of extension groups in arbitrary abelian categories. This definition reduces to the usual approach with 
derived functors of $\Hom$ in case there are enough projective or injective objects.
For any abelian category $\cA$, two fixed objects $M,N\in \cA$ and $d\in\mathbb{N}$, 
the set $\Ext_{\cA}^d(M,N)$ is defined
as follows. Consider the set of all exact sequences of length $d+2$,
\begin{equation}\label{eq531}
\mathcal{X}:\quad 0\to N\to X_1\to X_{2}\to\cdots\to X_d\to M\to 0,
\end{equation}
with $X_1,\cdots,X_d\in\cA$. Take two exact sequences $\cX$ and $\cY$ of the above form. If there 
are morphisms $X_i\to Y_i$, for $1\le i\le d$, such that the following diagram commutes:
\begin{displaymath}
\xymatrix{ 
0\ar[r]&N\ar[r]\ar@{=}[d]&X_1\ar[r]\ar[d]&X_{2}\ar[r]\ar[d]&\dots\ar[r] &X_d\ar[r]\ar[d]&M\ar[r]\ar@{=}[d]&0\\
0\ar[r]&N\ar[r]&Y_1\ar[r]&Y_{2}\ar[r]&\dots\ar[r] &Y_d\ar[r]&M\ar[r]&0,
} 
\end{displaymath}
we set $\cX\sim \cY$. Then $\Ext^d_{\cA}(M,N)$ is the set of equivalence classes of such exact sequences with
respect to the equivalence relation generated by $\sim$. This set has the natural structure of an abelian group, see~\cite{Buchsbaum}.

By \cite[Theorem 3.1]{Buchsbaum}, for any short exact sequence $X\hookrightarrow Y\tto Z$ 
in $\cA$ and any $K\in\cA$, there is the familiar long exact sequence 
\begin{equation}\label{resultBuch}
\begin{array}{ccccccccc}
0&\to &\Hom_{\cA}(Z,K)&\to &\Hom_{\cA}(Y,K)&\to&\Hom_{\cA}(X,K)&\to& \\
&&\Ext^1_{\cA}(Z,K)&\to &\Ext^1_{\cA}(Y,K)&\to&\Ext^1_{\cA}(X,K)&\to& \\
&&\Ext^2_{\cA}(Z,K)&\to &\Ext^2_{\cA}(Y,K)&\to&\Ext^2_{\cA}(X,K)&\to&\dots \\
\end{array}
\end{equation}
and similarly with $K$ being the first argument. 

By a result of Verdier, another possible introduction of the Yoneda extensions is via the derived category. For an arbitrary abelian category $\cA$ and two objects $M,N\in \cA$ we have
\begin{equation}
\label{eqVerdier}
\Ext^i_{\cA}(M,N)\;\cong\;\Hom_{\cD^b(\cA)}(M,N[i]),
\end{equation}
see Proposition 3.2.2 of \cite{Ve}, with $\cD^b(\cA)$ the bounded derived category and $N[i]$ the complex in $\cD^b(\cA)$ obtained by shifting $N$ by $i$ positions to the left.

\subsection{Extension full subcategories}
\label{secPrel.2}

Consider an abelian category $\cA$ and an abelian full subcategory $\cB$. 
Assume that the inclusion functor $\iota:\cB\to \cA$ is exact. By definition, 
the inclusion of $\cB$ into $\cA$ induces the canonical morphism of extension groups,
$$\varphi^d_{M,N}\,:\,\,\,\Ext_{\cB}^d(M,N)\to \Ext_{\cA}^d(M,N),$$
for any two objects $M,N\in\cB$ and any $d\in\mathbb{N}$. For convenience, we will 
leave out the reference to $M,N$ and when we say that a property holds for $\varphi^d$, it is 
understood that it holds for any $\varphi^d_{M,N}$. In general the morphisms $\varphi^d_{M,N}$
are neither injective nor surjective. 

We say that $\cB$ is {\em extension full} in $\cA$ if and only if $\varphi^d$ is an isomorphism for every $d\in\mN$.
Note that $\varphi^0$ is always an isomorphism since $\cB$ is a full subcategory of $\cA$, while 
$\varphi^1$ is an isomorphism if $\cB$ is assumed to be a Serre subcategory of $\cA$.
For convenience, we will slightly abuse notation and often write $\Ext_{\cB}^d(M,N)\cong \Ext_{\cA}^d(M,N)$ 
to state the  more specific  property that $\varphi^d_{M,N}$ is an isomorphism (in other words, we always
assume that, if $\Ext_{\cB}^d(M,N)$ and $\Ext_{\cA}^d(M,N)$ are isomorphic, then this isomorphism
is induced by $\varphi^d_{M,N}$). 

We will often use the following easy observation
which follows directly from the definitions using \cite[Theorem~3.1]{Buchsbaum} and 
\cite[Lemma~III.1.4]{Maclane}.

\begin{remark}
\label{remBuchphi}{\rm
The maps $\varphi^d$ give rise to a morphism (i.e. a chain map) between the corresponding long exact sequences of 
the form \eqref{resultBuch} with respect to categories $\cB$ and $\cA$.
}
\end{remark}

\subsection{Projective and global dimension}
\label{secPrel.3}

For $M\in\cA$ the {\em projective dimension} $\pd_{\cA}M\in\mN\cup\{+\infty\}$ of $M$ is the supremum of the set of all 
$k\in\mN$ for which there exists an $N\in\cA$ such that $\Ext^k_{\cA}(M,N)\not=0$. If the category $\cA$ 
contains enough projective objects, the projective dimension of $M\in\cA$ coincides with the minimal length 
of a projective resolution of $M$ in $\cA$. The supremum of all the projective dimensions over all objects 
in $\cA$ is called the {\em global dimension} of $\cA$ and is denoted by $\mathrm{gl.dim}\, \cA$.

Given a short exact sequence $A\hookrightarrow B\tto C$ with $A,B,C\in\cA$, 
the long exact sequence \eqref{resultBuch} implies the following inequalities:
\begin{eqnarray}
\label{pd1}
\pd_{\cA}A&\le & \max\{\pd_{\cA}B\,,\,\pd_{\cA}C\,-1\};\\
\label{pd2}
\pd_{\cA}C&\le &\max\{\pd_{\cA}A\,+1\,,\, \pd_{\cA}B\};
\end{eqnarray}
where $\le$ is the natural order on $\mZ\cup\{+\infty\}$.

\subsection{Guichardet categories}
\label{secPrel.4}

Consider an abelian category $\cA$ of finite global dimension and let $S_{\cA}$ denote the class
of simple objects in $\cA$. An {\em initial segment} in $\cA$ is the Serre subcategory $\cI$  of 
$\cA$ generated by a subset $S_{\cI}\subset S_{\cA}$, for which the following condition is satisfied:
for any $L,L'\in S_{\cA}$ such that $\pd_{\cA}L'= \pd_{\cA}L-1$, $L\in S_{\cI}$
and $\Ext_{\cA}^1(L,L')\not=0$, we have $L'\in S_{\cI}$.

An abelian category $\cA$ of finite global dimension is called a {\em Guichardet category} 
if every initial segment $\cI$ is extension full in $\cA$.

An easy example of a Guichardet category is the category $\cA$ of modules over the following quiver with relations:
\begin{equation*}
\xymatrix{ 
1\ar@/^/[rr]^{a}&&2\ar@/^/[ll]^{b}
}\qquad ab=0.\end{equation*}
The simple modules corresponding to the vertices satisfy $\pd_{\cA}L_1=1$ and $\pd_{\cA}L_2=2$. 
The only non-trivial initial segment is thus the Serre subcategory generated by $L_1$. 
This is a semi-simple category with unique
(up to isomorphism) simple object~$L_1$, as $\Ext^1_{\cA}(L_1,L_1)=0$. 
This subcategory is clearly extension full as $\Ext^k_{\cA}(L_1,L_1)=0$ for all $k>1$. 
This example corresponds, of course, to the principal block in category $\cO$ for~$\mathfrak{sl}(2)$,
see e.g. \cite[Theorem~5.31]{MazNN}.

We also provide an easy example of a category which is not Guichardet, another one can be
found in Subsection \ref{secsing.2}. Consider the category $\cA$ of modules over the 
following quiver with relations:
\begin{equation*}
\xymatrix{ 
1\ar@/^/[rr]^{a}&&2\ar@/^/[rr]^{c}\ar@/^/[ll]^{b}&&3\ar@/^/[ll]^{d}
}\qquad cd=0,\,\, ab=0.
\end{equation*}
Then it follows easily that we have a minimal projective resolution
\begin{equation}\label{20150607}
0\to P_3\to  P_2\to P_3\to L_3\to 0.
\end{equation}
So $\pd\, L_3=2$ and we similarly find $\pd\, L_1=1$ and $\pd\, L_2=2$. The Serre subcategory generated 
by $L_1$ and $L_3$ is thus an initial segment and is semi-simple (as there are no arrows between $1$
and $3$ in the quiver).  However 
$\Ext_{\cA}^2(L_3,L_3)\not=0$ by \eqref{20150607}, which contradicts 
the possibility that the initial segment would be extension full.

\subsection{Various categories of Lie algebra modules}
\label{secPrel.5}

Let $\fg$ be a finite dimensional semisimple complex Lie algebra and $U(\fg)$ be its universal enveloping algebra.
Denote by $\fb$ a Borel subalgebra of $\fg$ with Cartan subalgebra $\fh$ and nilradical $\fn$. Denote by $I'$ an ideal of finite codimension in the local ring $S(\fh)_{(0)}$. The corresponding ideal in $S(\fh)=U(\fh)$ is denoted by $I=S(\fh)\cap I'$.  Consider the following categories of $\fg$-modules,
see e.g. \cite{BGG, MR2428237, SoergelF, Soergel}:
\begin{itemize}
\item $\fg$-mod: The category of finitely generated $U(\fg)$-modules.
\item $\cW^\infty$: The full subcategory of $\fg$-mod consisting of {\em generalized weight modules}; that
is modules on which the action of $\fh$ is locally finite.
\item $\cW^I$: The full subcategory of $\cW^\infty$-mod consisting of modules for which the nilpotent part of the $\fh$-action factors over $S(\fh)/I$ (note that this requires adjustment of weights for each generalized weight space).
\item $\cO^\infty$: The full subcategory in $\cW^\infty$ of locally $U(\fb)$-finite modules.
\item $\cO^I$: The full subcategory in $\cW^I$ of locally $U(\fb)$-finite modules.
\item $\cH$: The category of finitely generated $\fg$-bimodules which are 
locally finite for the adjoint action of $\fg$.
\item ${}^k_{\chi}\cH^l_{\theta}$: The full subcategory in $\cH$ of bimodules
which are annihilated by $(\ker \chi)^k$ on the left and by $(\ker\theta)^l$ on the right for two central characters $\chi$ and $\theta$.
\item $\displaystyle {}^\infty_{\chi}\cH^l_{\theta}\,=\,\bigcup_{k\in\mN}\,\,{}^k_{\chi}\cH^l_{\theta}$;  
$\displaystyle \,\,{}^k_{\chi}\cH^\infty_{\theta}\,=\,\bigcup_{l\in\mN}\,\,{}^k_{\chi}\cH^l_{\theta}$;  
$\displaystyle \,\,{}^\infty_{\chi}\cH^\infty_{\theta}\,=\,\bigcup_{k\in\mN}\,\,{}^k_{\chi}\cH^\infty_{\theta}$.
\end{itemize}
In particular, we have 
$\displaystyle \cO^\infty=\bigcup_{I}\cO^I$ and 
$\displaystyle \cW^\infty=\bigcup_{I}\cW^I$. If $I$ is chosen to 
be the maximal ideal $\fm$ in $S(\fh)_{(0)}$, we have $\cO^{\fm}=\cO$, the BGG category from \cite{BGG}.
Similarly, $\cW^{\fm}=\cW$ is the category of finitely generated weight modules. Simple objects in $\cO^\infty$ 
coincide with simple objects in $\cO$. Objects of $\cO^\infty$ (and of $\cO^I$) have finite length, 
so these categories are both, artinian and noetherian. The categories $\cO^\infty$ and $\cW^\infty$ have neither 
injective nor projective modules. 

For each central character $\chi$ and every category $\mathcal{X}$ of $\mathfrak{g}$-modules defined above,
we denote by $\mathcal{X}_{\chi}$ the full subcategory of $\mathcal{X}$ consisting of all modules with
generalized central character $\chi$.

For $\lambda\in\mathfrak{h}^*$, we denote by $L(\lambda)\in\cO$ the simple highest weight module with highest 
weight $\lambda$ and by $\chi_\lambda$ the central character of $L(\lambda)$.

\subsection{Restricted duality}
\label{secPrel.6}

We conclude by recalling the usual construction of {\em duality} (i.e. a contravariant exact involutive equivalence) on category $\cO^I$. We use the transpose map $\tau$ on $\fg$ described in \cite[Section~0.5]{MR2428237}. The
map $\tau$ fixes $\fh$ pointwise and sends the root space $\fg_{\alpha}$ to $\fg_{-\alpha}$ for each root $\alpha$. 
The $\fg$-action on the classical dual module $M^\ast=\Hom_{\mC}(M,\mC)$ is given by 
$(g\alpha)(v)=\alpha(\tau(g)v)$ for $g\in\fg$, $v\in M$ and $\alpha\in M^\ast$.

For each $\mu\in\fh^\ast$ and $M\in\cO^I$, denote by $M^{\mu}$ the $\fh$-submodule consisting of all
generalized weight vectors for weight $\mu$. The dual module (through $\tau|_{\fh}=id$) is denoted by 
$(M^{\mu})^\ast$. The nilpotent part of the action of $S(\fh)$ on $\left(M^{\mu}\right)^\ast$ clearly factors over $I$ and, moreover, we have $(M^\ast)^\mu\cong \left(M^\mu\right)^\ast$. Then define
$$M^\star=\bigoplus_{\mu} (M^\mu)^\ast \in \cO^I,$$
canonically as a submodule of $M^\ast$. By definition, this leads to a duality $\star:\cO^\infty\to \cO^\infty$,
which also fixes every $\cO^I$. This duality also induces the usual duality on $\cO$ as in
\cite[Section~3.2]{MR2428237}. Similarly to \cite[Theorem~3.2(e)]{MR2428237} we have 
\begin{equation}
\label{extdual}\Ext^\bullet_{\cO^\infty}(M^\star,N^\star)\cong\Ext^\bullet_{\cO^\infty}(N,M)
\end{equation} 
for any two $M,N\in\cO^\infty$.

\subsection{Lie algebra cohomology}

For a finite dimensional Lie algebra $\fa$, the algebra cohomology of~$\fa$ with values in $M\in\fa$-mod satisfies
$$H^d(\fa,M)\cong\Ext^d_{\fa}(\mC,M)\quad\mbox{ for }d\in\mN,$$
see Corollary 7.3.6 in \cite{Weibel}. We will need the following simple lemma.
\begin{lemma}
\label{algcohom}
For any module $V\in\fa$-mod, we have
$$\dim H^{\dim\fa}(\fa,V)= \dim\Hom_{\fa}(V,\mC).$$
\end{lemma}
\begin{proof}
This can be proved by standard methods using the Chevalley-Eilenberg complex in \cite[Corollary~7.7.3]{Weibel} and the analogue of sequence~\eqref{resultBuch}, with $K=\mC$ in the first argument. It is also an immediate consequence of the general principle
$$H^p(\fa,V)\cong H_{\dim\fa-p}(\fa,V),$$
known as Poincar\'e duality for Lie algebra cohomology.
\end{proof}
 

\section{Criteria for extension fullness}
\label{secExtFull}
In this section we will derive some useful criteria for extension fullness. 
Our setup consists of an abelian category $\cA$ and a full abelian subcategory $\cB$.
Further, we always assume that the inclusion functor $\iota:\cB\to \cA$ is exact.

\begin{lemma}\label{5lem1}
Let $\cA$ and $\cB$ be as above. If all objects of $\cB$ have finite length, 
then $\cB$ is extension full in $\cA$ if and only if 
$$\varphi^d_{L,L'}:\Ext^d_{\cB}(L,L')\to \Ext^d_{\cA}(L,L')$$
is an isomorphism for any two simple objects $L,L'\in\cB$ and any $d\in\mathbb{N}$.
\end{lemma}

\begin{proof}
The ``only if'' statement is clear.  We prove  the ``if'' statement by induction on the length of an object in $\cB$. 
Assume that we have $\Ext^d_{\cB}(M,L')\cong \Ext^d_{\cA}(M,L')$ for all $d\in\mN$ and for any simple $L'\in \cB$ and $M\in \cB$ of length smaller that or equal to $i-1$.  The module $M$ admits a short exact sequence $N\hookrightarrow M\tto K$ where $N,K\in\cB$ have length smaller than $i$. 

Consider the chain map induced by $\varphi^d$ between the long exact sequences of the form \eqref{resultBuch}
constructed with respect to both of the categories $\cA$ and $\cB$ (see Remark~\ref{remBuchphi}).
Now the isomorphism $$\Ext^d_{\cB}(M,L')\to\Ext^d_{\cA}(M,L')$$ follows from the Five Lemma 
(see e.g. \cite[Lemma~I.3.3]{Maclane}).

Now the proof that $L'$ can also be replaced by an arbitrary object of $\cB$ is similar.
\end{proof}

\begin{lemma}\label{5lem2}
Let $\cA$ and $\cB$ be as above. Assume that $\cB$ has a full subcategory $\cB^0$ with the following properties
\begin{itemize}
\item $\cB$ is the Serre subcategory of $\cA$ generated by the objects of $\cB^0$
\item $\cB^0$ has enough projective objects. 
\end{itemize}
Then $\cB$ is extension full in $\cA$ if and only if, for $d\in\mathbb{N}$, the map 
\begin{equation*}
\Ext_{\cB}^d(P,K)\to \Ext_{\cA}^d(P,K)
\end{equation*}
is an isomorphism for every projective $P$ in $\cB^0$ and every $K\in\cB^0$.
\end{lemma}

\begin{proof}
The ``only if'' statement is clear, so we prove  the ``if'' statement.

We start by proving, by induction on $d$, that $\varphi^d_{M,K}$ is always a monomorphism for arbitrary $M,K\in\cB^0$. Since $\cB$ is a Serre subcategory of $\cA$, $\varphi^1$ is an isomorphism. Now we assume that $\varphi^i$ (restricted to $\cB^0$) is a monomorphism for $i<d$. Take arbitrary $C,K\in \cB^0$, then there is a $P$, projective in $\cB^0$, such that there is a short exact sequence $X\hookrightarrow P\tto C$ for some $X\in\cB^0$. From \eqref{resultBuch} 
and Remark~\ref{remBuchphi} we have the following commutative diagram with exact rows:
\begin{displaymath}
\xymatrix{ 
\Ext^{d-1}_{\cB}(P,K)\ar[r]\ar[d]_{\varphi^{d-1}_{P,K}}
&\Ext^{d-1}_{\cB}(X,K)\ar[r]\ar[d]_{\varphi^{d-1}_{X,K}}
&\Ext^d_{\cB}(C,K)\ar[r]\ar[d]_{\varphi^d_{C,K}}
& \Ext^d_{\cB}(P,K)\ar[d]_{\varphi^d_{P,K}}\\
\Ext^{d-1}_{\cA}(P,K)\ar[r]&\Ext^{d-1}_{\cA}(X,K)\ar[r]&\Ext^d_{\cA}(C,K)\ar[r]& \Ext^d_{\cA}(P,K)\\
}
\end{displaymath}
Now, by assumption, $\varphi^{d-1}_{P,K}$ and $\varphi^d_{P,K}$ are isomorphisms and from the 
induction step $\varphi^{d-1}_{X,K}$ is a monomorphism. The Four Lemma (see e.g \cite[Lemma~I.3.3(i)]{Maclane})
therefore implies that $\varphi^d_{C,K}$ is injective, for arbitrary $C,K\in\cB^0$. 

Now we prove, by induction on $d$, that $\varphi^d_{C,K}$ is actually an isomorphism. Assume $\varphi^i$ is an isomorphism for $i<d$ and consider $P$ and $X$ as in the paragraph above. From \eqref{resultBuch} 
and Remark~\ref{remBuchphi} we have the following commutative diagram with exact rows:
\begin{displaymath}
\xymatrix{ 
\Ext^{d-1}_{\cB}(X,K)\ar[r]\ar[d]_{\varphi^{d-1}_{X,K}}
&\Ext^{d}_{\cB}(C,K)\ar[r]\ar[d]_{\varphi^{d}_{C,K}}
&\Ext^d_{\cB}(P,K)\ar[r]\ar[d]_{\varphi^d_{P,K}}
& \Ext^d_{\cB}(X,K)\ar[d]_{\varphi^d_{X,K}}\\
\Ext^{d-1}_{\cA}(X,K)\ar[r]&\Ext^{d}_{\cA}(C,K)\ar[r]&\Ext^d_{\cA}(P,K)\ar[r]& \Ext^d_{\cA}(X,K)\\
}
\end{displaymath}
As $\varphi^{d-1}_{X,K}$ is a bijection by the induction step, $\varphi^d_{P,K}$ is a bijection by assumptions, and $\varphi^d_{X,K}$ is a monomorphism by the previous paragraph, the Four Lemma implies that $\varphi^d_{C,K}$ is an epimorphism.

By assumptions, any module in $\cB$ has a finite filtration with quotients in $\cB^0$. The claim of the
lemma now follows using the same argument as in the proof of Lemma~\ref{5lem1}.
\end{proof}

The following result is a special case of Lemma \ref{5lem2}, but we provide an alternative proof, which is
of interest in its own right. Note also the connection with Theorem 3.9 of \cite{Ps}.

\begin{corollary}\label{genprop}
Let $\cA$ and $\cB$ be as above and assume that
they both have enough projective objects. If every projective object in $\cB$ is acyclic 
for the functor $\Hom_{\cA}(-,K)$ for any $K\in\cB$, then $\cB$ is extension full in $\cA$.
\end{corollary}

\begin{proof}
Consider $N\in\cB$ fixed. We need to prove that the functor $\Ext^j_{\cA}(-,N)$, restricted to category $\cB$, is isomorphic to $ \Ext^j_{\cB}(-,N)$. We have the obvious isomorphism
$$\Hom_{\cB}(-,N)\cong\Hom_{\cA}(-,N)\circ \iota,$$
of functors from the category $\cB$ to the category $\mathbf{Sets}$.

By assumption, the exact functor $\iota$ maps projective modules in $\cB$ to acyclic modules for the functor $\Hom_{\cA}(-,N)$. The classical Grothen\-dieck spectral sequence, see \cite[Section~5.8]{Weibel}, therefore implies the theorem.
\end{proof}

Now we consider an extra abelian category $\cC$, for which $\cA$ (and therefore also $\cB$) is a full subcategory
with exact inclusion. Denote by $\cA^\infty$ the Serre subcategory of $\cC$ generated by objects of $\cA$. Furthermore we denote by $\cA^k$, with $k\in\mN$, the subcategory of $\cA^\infty$ of objects which have a filtration of length $k$ with quotients inside $\cA$, then $\displaystyle \cA^\infty=\bigcup_{k\in\mN}\cA^k$. We define similarly $\cB^k$ and $\cB^\infty$ as subcategories in $\cC$, which are automatically subcategories of $\cA^k$ and $\cA^\infty$, respectively.

We show how the Yoneda extension groups in $\cA^\infty$, see Subsection~\ref{secPrel.1}, can be seen as a direct limit of the corresponding extension groups for~$\cA^k$.
\begin{proposition}
\label{extinfty}
\begin{enumerate}[$($i$)$]
\item\label{extinfty.1} For $M,N\in\cA^\infty$, $\Ext^d_{\cA^\infty}(M,N)$ corresponds the space of exact sequences
$$0\to N\to X_1\to X_2\to\cdots\to X_d\to M\to 0,$$
where all modules $M,N,X_1,\cdots,X_d$ are contained in $\cA^k$ for some~$k$. Two such exact sequences are equivalent if and only if they represent the same extension in $\cA^l$ for some  $l\geq k$.
\item\label{extinfty.2} The extension groups $\Ext^d_{\cA^\infty}(M,N)$ where $M$ and $N$ are taken in some
$\cA^{k_0}\subset\cA^\infty$ correspond to the limit of the directed system
$$\Ext^d_{\cA^{k}}(M,N)\to \Ext^d_{\cA^{k+1}}(M,N),\qquad k\geq k_0,$$
where these morphism are, in general, neither injective nor surjective.
\end{enumerate}
\end{proposition}

\begin{proof}
By definition of $\cA^\infty$, for any finite exact sequence $\cX^\bullet$ of objects in $\cA^\infty$ there must be some finite $k$ such that all appearing modules are objects of $\cA^k$. We denote the minimal such $k$ by $k(\cX^\bullet)$. Now consider two exact sequences $\cX^\bullet$ and $\cY^\bullet$ of length $d+2$ which start with $N$ and end with $M$. By Subsection \ref{secPrel.1} they are equivalent inside $\Ext^d_{\cA^\infty}(M,N)$ if and only if there exists a finite number of exact sequences $\cX_1,\cdots,\cX_{p}$ as above, which have appropriate morphisms between them. When taking 
$$l:=\max\{k(\cX^\bullet);\,k(\cY^\bullet);\, k(\cX_i)\, i=1..p \}$$
we find that $\cX^\bullet\sim\cY^\bullet$ inside $\Ext^d_{\cA^l}(M,N)$. By definition, as soon as  $\cX^\bullet\sim\cY^\bullet$ inside $\Ext^d_{\cA^l}(M,N)$ for some $l$, they are equivalent inside $\Ext^d_{\cA^\infty}(M,N)$. This concludes the proof of part~\eqref{extinfty.1}.

Part~\eqref{extinfty.2} is a reinterpretation of part~\eqref{extinfty.1}. That the morphism need not be surjective follows from the fact that even when two exact sequences $\cX^\bullet$ and $\cY^\bullet$ as above with all objects in $\cA^k$ are not equivalent in $\Ext^d_{\cA^k}(M,N)$, there can be a third exact sequence $\cZ^\bullet$ with objects in $\cA^{k+1}$ such that $\cX^\bullet\sim \cZ^\bullet$ and $\cY^\bullet\sim \cZ^\bullet$ and hence $\cX^\bullet$ becomes equivalent to $\cY^\bullet$ inside $\Ext^d_{\cA^{k+1}}(M,N)$. On the other hand, there can also be an exact sequence $\cW^\bullet$ with objects in $\cA^{k+1}$ which is not equivalent (inside $\Ext^d_{\cA^{k+1}}(M,N)$) to any exact sequence with objects in $\cA^k$ implying the morphism is not necessarily surjective.
\end{proof}

\begin{corollary}
\label{inftyabstract}
Consider abelian categories $\cB\subset\cA\subset\cC$, with $\cB^k$ and $\cA^k$ as defined above. If $\cB^k$ is extension full in $\cA^k$ for each $k$, then $\cB^\infty$ is extension full in $\cA^\infty$.
\end{corollary}

\begin{proof}
To prove the isomorphism 
$$\Ext_{\cB^\infty}^d(M,N)\cong \Ext_{\cA^\infty}^d(M,N) $$
for every $M,N\in\cB^\infty$ and $d\in\mathbb{N}$, we need to prove two statements according to 
Proposition~\ref{extinfty}(ii). 

{\bf Statement I.} Every exact sequence of the form \eqref{eq531},
where all modules are contained in some $\cA^k$ and which is not a trivial extension in any of the categories $\cA^l$ for $l\ge k$, is equivalent to an extension in $\cB^\infty$.

{\bf Statement II.}  Every exact sequence of the form \eqref{eq531},
where all modules are contained in some $\cB^k$ and which is not a trivial extension in any of the categories $\cB^l$ for $l\ge k$, does not become a trivial extension in $\cA^\infty$.

We prove Statement I. By assumption, the extension given by \eqref{eq531} is equivalent to one in $\cB^k$. Since the same extension is not trivial in $\cA^l$ for an arbitrary $l \ge k$, it is also a non-trivial extension in $\cB^l$. This proves that this extension is equivalent to a non-trivial extension in $\cB^\infty$. Statement~II is proved similarly.
\end{proof}

We conclude this section with the observation that the condition of extension fullness is equivalent to a seemingly stronger condition. Note that this property does not require the existence of projective or injective objects.
\begin{proposition}
The full abelian subcategory $\cB$ of $\cA$ with exact inclusion $\imath$ is extension full if and only if $\imath$ induces a full and faithful exact (i.e. triangulated) functor
$$\imath:\;\cD^b(\cB)\,\to\,\cC^b(\cA).$$
\end{proposition}
\begin{proof}
By equation \eqref{eqVerdier}, extension fullness is equivalent to the condition that $\imath$ induces an isomorphism
$$\imath:\;\Hom_{\cD^b(\cB)}(\cX^\bullet,\cY^\bullet)\,\;\tilde\to\;\,\Hom_{\cD^b(\cA)}(\imath(\cX^\bullet),\imath(\cY^\bullet))$$
for all complexes $\cX^\bullet,\cY^\bullet$ in $\cD^b(\cB)$, each of which is concentrated
in a single position. It hence suffices to prove that if this displayed equation is true for all such
complexes, then it is true for arbitrary complexes in $\cD^b(\cB)$. This can be proved similarly to Lemma~\ref{5lem1}. It suffices to replace short exact sequences by distinguished triangles, using the long exact sequences in III.1.2.5 (see also Proposition II.1.2.1) of \cite{Ve} or Example 10.2.8 of \cite{Weibel} and note that the bounded derived category is generated, as a triangulated category, by complexes concentrated
in a single position. 
\end{proof}


\section{Category $\cO$ and weight modules}
\label{secO}
The main result of this section is stated in the following theorem.

\begin{theorem}\label{main}
Let $\fg$ be a semisimple finite dimensional complex Lie algebra and $I'$ 
an ideal of finite codimension in $S(\fh)_{(0)}$.
\begin{enumerate}[$($i$)$]
\item\label{main.1}  The category $\cO^I$ is extension full in $\cW^I$.
\item\label{main.2}  The category $\cO^\infty$ is extension full in $\cW^\infty$.
\end{enumerate}
\end{theorem}

First we note that $\cO^I$ is a Serre subcategory of $\cW^I$ and hence for $M,N\in\cO^I$ we have
\begin{equation}\label{OKdeg1}
\Ext_{\cO^I}^1(M,N)\cong\Ext_{\cW^I}^1(M,N).
\end{equation}
For each $k\in\mN$, we can consider the ideal $I_k=\fh^k S(\fh)$, for which we use the short-hand notation $\cO^k=\cO^{I_k}$ and $\cW^k=\cW^{I_k}$. Then we arrive in the situation described at the end of 
Section~\ref{secExtFull}, with $\displaystyle \cW^\infty=\bigcup_{k\in\mathbb{N}} \cW^k$ and 
$\displaystyle \cO^\infty=\bigcup_{k\in\mathbb{N}}\cO^k$.  For notational 
convenience, we will work with the ideals $I_k$ even though the results hold generally.

For every $k>0$ and $\lambda\in\fh^\ast$, we define the $\fh^k$-module $V_{\lambda,k}$ as
$U(\fh)/J_{\lambda,k}$ where $J_{\lambda,k}$ is the ideal of $U(\fh)$ generated by all elements
of the form $(h_1-\lambda(h_1))(h_2-\lambda(h_2))\cdots(h_k-\lambda(h_k))$ where $h_1,h_2,\dots,h_k\in\fh$.
Further, for $n>0$, we define the $\fg$-module
$$M_{n,k}(\lambda):= 
U(\fg)\,\bigotimes_{U(\fb)}\,U(\fb)/U(\fb)\fn^n\,\bigotimes_{U(\fh)}\, V_{\lambda,k}.$$
The following lemma is immediate from the definition.

\begin{lemma}\label{firstly}
\begin{enumerate}[$($i$)$]
\item\label{firstly.1}
We have $M_{n,k}(\lambda)\in \cO^k$.
\item\label{firstly.2}
Both $M_{n+1,k}(\lambda)$ and $M_{n,k+1}(\lambda)$ surject onto $M_{n,k}(\lambda)$. Furthermore, the module
$M_{1,1}(\lambda)$ is isomorphic to the classical Verma module $M(\lambda)$ with highest weight $\lambda$.
\item\label{firstly.3}
There is the following short exact sequence:
\begin{displaymath}
U(\fg)\fn^n\bigotimes_{U(\fh)} V_{\lambda,k}\hookrightarrow U(\fg)\bigotimes_{U(\fh)}V_{\lambda,k}\tto M_{n,k}(\lambda).  
\end{displaymath}
\end{enumerate}
\end{lemma}

We denote the maximal direct summand of $M_{n,k}(\lambda)$ belonging to the subcategory $\cO^k_{\chi_\lambda}$, by $\widetilde{M}_{n,k}(\lambda)$.  

\begin{lemma}\label{secondly}
For each $d>0$ we have $\Ext^d_{\cW^k}(\widetilde{M}_{n,k}(\lambda),L(\nu))=0$
for all $\nu\in\fh^\ast$ and all  $n\gg 0$.
\end{lemma}

\begin{proof}
The result is trivial unless $\chi_\nu=\chi_\lambda$, so we assume that $\nu$ is in the Weyl group orbit of $\lambda$. This leaves only a finite amount of choices for $\nu$.

The module $U(\fg)\otimes_{U(\fh)}V_{\lambda,k}$ is projective in $\cW^k$. 
Lemma~\ref{firstly}\eqref{firstly.3} and equation \eqref{resultBuch} therefore imply that we have 
\begin{equation}\label{eqpro} 
\Ext^{d-1}_{\cW^k}\big(U(\fg)\fn^n\otimes_{U(\fh)} V_{\lambda,k}
, L(\nu)\big)\tto \Ext^d_{\cW^k}(M_{n,k}(\lambda), L(\nu)),
\end{equation}
moreover,  this is an isomorphism if $d>1$. 

The $\fb$-module $U(\fb)\fn^n\otimes_{U(\fh)} V_{\lambda,k}$ has a resolution in terms of modules 
$U(\fb)\otimes_{U(\fh)} V_{\mu,k}$, where each $\mu\in\fh^\ast$ is of the form $$\mu=\lambda+\alpha_1+\cdots+\alpha_p,$$ where $p\ge n$ and $\alpha_i$'s are positive roots. This implies that the module
\begin{displaymath}
U(\fg)\fn^n\bigotimes_{U(\fh)} V_{\lambda,k}\cong
U(\fg)\bigotimes_{U(\fb)}U(\fb)\fn^n\bigotimes_{U(\fh)} V_{\lambda,k}
\end{displaymath}
has a projective resolution in $\cW^k$, by modules
$U(\fg)\bigotimes_{U(\fh)} V_{\mu,k}$, with the same condition on~$\mu$. 

According to the above, in order to prove that the left-hand side of equation \eqref{eqpro} is zero for $d>0$, it suffices to show that the space
$$\Hom_{\fg}\left(U(\fg)\otimes_{U(\fh)} V_{\mu,k},L(\nu)\right)\,\cong\, 
\Hom_{\fh}\left(V_{\mu,k},L(\nu)\right)$$
(where the isomorphism is given by adjunction)
is zero, for $\mu$ as above. For each $\nu$, we can find an $n$ large enough, such that all of weights $\mu$ 
of the above form do not appear in $L(\nu)$. Taking the maximum over this finite set of numbers yields the lemma.
\end{proof}

For every $\mu\in\fh^\ast$, denote the projective cover of $L(\mu)$ in $\cO^k$ by $P^k(\mu)$.
Let $W$ be the Weyl group of $\mathfrak{g}$.

\begin{proposition}\label{projexp}
For $n$ large enough, we have
$$\widetilde{M}_{n,k}(\lambda)=\bigoplus_{w\in W} \dim(L(w\cdot\lambda)^\lambda)\, P^k(w\cdot\lambda).$$
\end{proposition}

\begin{proof}
Lemma~\ref{secondly} for $d=1$ and the isomorphism in  \eqref{OKdeg1} imply that $\widetilde{M}_{n,k}(\lambda)$ is projective in $\cO^k$, for $n$ large enough. Furthermore, from Lemma~\ref{firstly}\eqref{firstly.3} 
and the computation in the proof of Lemma~\ref{secondly} we get 
$$\Hom_{\fg}\big(\widetilde{M}_{n,k}(\lambda),L(w\cdot\lambda)\big)\cong
\Hom_\fh\big(V_{\lambda,k},L(w\cdot\lambda)\big),$$
which concludes the proof.
\end{proof}

\begin{corollary}
\label{proj}
Consider $M\in\cO^I$ and $P$ projective in $\cO^I$. Then for $d>0$ we have
$$\Ext_{\cW^I}^d(P,M)=0.$$
\end{corollary}
\begin{proof}
If $M$ is simple, this is an immediate consequence of the combination of Lemma \ref{secondly} and Proposition \ref{projexp}. The general statement therefore follows, using the usual arguments with long exact sequences,
from the fact that each module in $\cO^I$ has finite length.
\end{proof}

\begin{proof}[Proof of Theorem \ref{main}]
Claim~\eqref{main.1} follows combining Corollaries~\ref{genprop} and \ref{proj}. 
Claim~\eqref{main.2} follows from claim~\eqref{main.1} and Corollary~\ref{inftyabstract}.
\end{proof}

For $I=\fm$ the category $\cO^I$ is the usual BGG category $\cO$. In this case Theorem \ref{main}\eqref{main.1} 
states that category $\cO$ is extension full in the ca-tegory of weight modules, which recovers an old result 
of Delorme, see \cite{De}. An important consequence is the following connection between $\fn$-cohomology and extensions 
with Verma modules in category $\cO$, see \cite[Theorem~6.15(b)]{MR2428237}. 

\begin{corollary}\label{cor458}
For $\mu\in\fh^\ast$ and $N\in{\cO}$ we have
$$\Hom_{\fh}(\mC_\mu, H^k(\fn,N))=\Ext^k_{{\cO}}(M(\mu),N).$$
\end{corollary}

\begin{proof}
The equality $\Hom_{\fh}(\mC_\mu, H^k(\fn,N))=\Ext^k_{{\cW}}(M(\mu),N)$ follows immediately from the Frobenius reciprocity. The claim thus follows from Theorem~\ref{main}(i) for the case $I=\fm$.
\end{proof}

We would like to record the following observation.

\begin{proposition}\label{gldim}
We have 
\begin{eqnarray*}
{\rm gl.dim}\,\cO \,=\, {\rm gl.dim}\,\cW&=&\dim \fg-\dim\fh,
\end{eqnarray*}
whereas the global dimensions of $\cO^I$ and $\cW^I$ are infinite if $I\not=\mm$.
\end{proposition}

\begin{proof}
The global dimension of $\cO$ is well-known, see e.g. \cite[Proposition~2]{MR2366357}, 
\cite[Section~6.9]{MR2428237} or \cite{BGG}. The global dimension of $\cW$ follows from the 
fact that a projective resolution in $\cW$ is a projective resolution for the relative $(\fg,\fh)$-cohomology.

The infinite global dimensions of $\cO^I$ and $\cW^I$ follow immediately from considering a projective resolution in $\cO^I$ (respectively $\cW^I$) of a projective module in $\cO$ (respectively $\cW$).
\end{proof}

\section{Category $\cO^\infty$}\label{secOinf}

\subsection{Category $\cO^\infty$ is extension full in $\mathfrak{g}\text{-}\mathrm{mod}$}\label{secOinf.1}

\begin{theorem}\label{finalresult}
Let $\fg$ be a complex semisimple finite dimensional Lie algebra. Then both categories $\cO^\infty$ and 
$\cW^\infty$ are extension full in $\mathfrak{g}\text{-}\mathrm{mod}$.
\end{theorem}

Before proving this, we note the following corollary.

\begin{corollary}\label{gldiminfty}
We have  ${\rm gl.dim}\,\cO^\infty\,=\,{\rm gl.dim}\,\cW^\infty\,=\,\dim\fg$.
\end{corollary}

\begin{proof}
Theorem \ref{finalresult} implies that the global dimension of $\cO^\infty$ and $\cW^\infty$ are smaller than
or equal to $\dim\fg$, the global dimension of $\mathfrak{g}\text{-}\mathrm{mod}$. 
The classical fact  $$\Ext^{\dim\fg}_{\fg}(\mC,\mC)=H^{\dim\fg}(\fg,\mC)\not=0,$$
see Lemma \ref{algcohom}, then shows that both global dimensions in question are equal to this value.
\end{proof}

The remainder of this section is devoted to proving Theorem \ref{finalresult}.

\begin{lemma}\label{lemabel}
Consider a finite dimensional abelian Lie algebra $\fh$ and the category 
$\displaystyle \cF:=\bigcup_{k\in\mN}\cF^k$ where $\cF^k$ is the category of all finite dimensional $\fh$-modules 
which are $S(\fh)/\fh^kS(\fh)$-modules. Then the category $\cF$ is extension full in $\fh$-{\rm mod}.
\end{lemma}

\begin{proof} 
The commutative algebra $S(\fh)$ is positively graded in the natural way with $\fh$ being of degree one.
The algebra $S(\fh)$, being isomorphic to the polynomial algebra in finitely many variables, is Koszul.
Consider the graded Koszul resolution $\mathcal{P}^{\bullet}$ of the trivial $S(\fh)$-module
$\mC$ (the module structure on $\mC$ is given by $\fh\mC=0$). Then the $-i$-th component $\mathcal{P}^{-i}$
of this resolution is generated in degree $i$ and $\mathcal{P}^{-i}=0$ for $i>\dim \fh$.

For $k\in\mathbb{N}$ consider the algebra $A_k:=S(\fh)/\fh^kS(\fh)$ together with the functor
$F_k:S(\fh)\text{-}\mathrm{mod}\to A_k\text{-}\mathrm{mod}$ given by $M\mapsto M/\fh^kS(\fh)M$.
Applying $F_k$ to $\mathcal{P}^{\bullet}$ gives a complex of projective $A_k$-modules which still has
homology $\mC$ in the homological position zero and a lot of other homologies in negative homological positions.
However, all those homologies are concentrated in degrees $\geq k$ of our grading. Resolving those
homologies in  $A_k\text{-}\mathrm{mod}$ we obtain that for $k\gg 0$ the graded spaces
$\Ext_{\fh}^d(\mC,\mC)$ and $\Ext_{A_k}^d(\mC,\mC)$ agree in all degrees up to $k-1$ for all $d$.
Hence, taking the limit for $k\to \infty$, yields isomorphism $\Ext^d_{\cF}(\mC,\mC)\cong\Ext_{\fh}^d(\mC,\mC)$.

Since all modules in $\cF$ have finite length and $\mC$ is the only simple module in $\cF$, the result follows from Lemma \ref{5lem1}.
\end{proof}

Frobenius reciprocity for extensions follows from adjunction between derived functors. Since the category $\cW^\infty$ does not have projective modules, we need the following lemma. We introduce the notation $\cC(\fh)^I$ for the category of finite dimensional $\fh$-modules for which the nilpotent part of the $S(\fh)$-action factors over $I$, with $I$ an ideal as in Section~\ref{secPrel}. Furthermore, we set $\displaystyle\cC(\fh)^\infty=\bigcup_{I}\cC(\fh)^I$.

\begin{lemma} \label{lemFrob}
For $K\in\cC(\fh)^\infty$, $M\in\cW^\infty$ and $d>0$ we have
\begin{equation}\label{eq976}
\Ext^d_{\cW^{\infty}}(\ind^{\fg}_{\fh}K,M)\cong \Ext^d_{\cC(\fh)^\infty}(K,\res^{\fg}_{\fh}M). 
\end{equation}
\end{lemma}

\begin{proof}
There is an ideal $J$ big enough such that both $M$ and $\ind^{\fg}_{\fh}K$ belong to $\cW^J$. 
By Proposition~\ref{extinfty}(ii), both the left-hand side and the right-hand side of \eqref{eq976} are 
respectively given as limits of
$$\Ext^d_{\cW^{I}}(\ind^{\fg}_{\fh}K,M)\quad\mbox{and}\quad \Ext^d_{\cC(\fh)^I}(K,\res^{\fg}_{\fh}M)$$
over $I\supset J$. The isomorphism between these two extension groups for every $J$ follows from the usual 
Frobenius reciprocity.
\end{proof}

\begin{lemma}\label{projCgh}
Let $I'$ be an ideal in $S(\fh)_{(0)}$ of finite codimension. 
If $P\in \cW^I$ is projective and $M\in\cW^\infty$ is arbitrary, then the morphism
$$\Ext^d_{\cW^\infty}(P,M)\to \Ext^d_{\fg}(P,M),$$ 
is an isomorphism for all $d\ge 0$.
\end{lemma}

\begin{proof}
Without loss of generality we may take
$$P\cong U(\fg)\otimes_{S(\fh)}(S(\fh)/I\,\otimes \, \mC_\lambda),$$
where $\mC_\lambda=V_{\lambda,1}$ is the simple $1$-dimensional $\fh$-module corresponding to $\lambda$.
By Lemma \ref{lemFrob}, the proposed statement then reduces to 
\begin{equation}\label{frobh}\Ext^d_{\cC(\fh)^\infty}(S(\fh)/I\,\otimes \, \mC_\lambda,\res^{\fg}_{\fh}M)\cong\Ext^d_{\fh}(S(\fh)/I\,\otimes \, \mC_\lambda,\res^{\fg}_{\fh}M).\end{equation}
All modules in $\cC(\fh)^\infty$ decompose into generalized weight spaces and the category decomposes into equivalent blocks corresponding to different eigenvalues. It suffices to consider one block. The block corresponding to 0 is exactly $\cF$ from Lemma \ref{lemabel} and equation \eqref{frobh} is thus a consequence of that lemma.
\end{proof}

\begin{proof}[Proof of Theorem \ref{finalresult}]
We apply Lemma~\ref{5lem2}, with $\cA$, $\cB$ and $\cB^0$ given by, respectively, $\fg$-mod, $\cW^\infty$ and $\cW$. The fact that $\cW^\infty$ is extension full in $\fg$-mod is therefore a consequence of Lemma~\ref{projCgh} for the special case $I=\fm$.

The fact that $\cO^\infty$ is extension full in $\fg$-mod is then an immediate consequence of the result for $\cW^\infty$ and Theorem~\ref{main}\eqref{main.2}.
\end{proof}

\subsection{Projective dimensions in $\cO^\infty$}\label{secOinf.2}

In this section we calculate projective dimensions inside category $\cO^\infty$ for modules in $\cO^I$.

\begin{theorem}\label{pdthm}
\begin{enumerate}[$($i$)$]
\item\label{pdthm.1} Consider $M\in\cO^I$ for some ideal $I'$ in $S(\fh)_{(0)}$ of finite codimension, with $\pd_{\cO^I}M<\infty$, then
$$\pd_{\cO^\infty}M\,\,=\,\,\dim\fh\,+\,\pd_{\cO^I}M.$$
\item\label{pdthm.2} 
The minimal projective dimension of a module in $\cO^\infty$ is $\dim\fh$.
\end{enumerate}
\end{theorem}

Before proving this, we note that this results yields the projective dimension of all structural modules from 
a regular block of $\cO$ 
(that is simple, standard, costandard, tilting, injective, projective modules) inside the category $\cO^\infty$ 
by using the results in \cite{MR2366357, MR2602033}. In particular, we have the following corollary.

\begin{corollary} \label{pdcor}
Consider $\lambda\in\fh^\ast$ to be integral regular dominant. 
Then for $w\in W$ we have:
\begin{enumerate}[$($a$)$]
\item $\pd_{\cO^\infty} L(w\cdot \lambda)=\dim\fg-l(w)$,
\item $\pd_{\cO^\infty} M(w\cdot\lambda)=\dim\fh+l(w)$.
\end{enumerate}
\end{corollary}

\begin{proof}
This is a consequence of the combination of Theorem~\ref{pdthm} with either  \cite[Theorem~6.9]{MR2428237} or   \cite[Propositions~3 and 6]{MR2366357}.
\end{proof}

The remainder of this subsection is devoted to proving Theorem~\ref{pdthm}.

\begin{lemma} \label{helppd}
For a projective  $P\in \cO^I$ we have $\pd_{\cO^\infty}P=\dim\fh$ and
\begin{equation}\label{extrext}\dim\Ext^{\dim\fh}_{\cO^\infty}(P^I(\lambda),L(\mu))=n_I \;\delta_{\lambda,\mu}\quad\forall\; \,\lambda,\mu\in\fh^\ast,
\end{equation}
for $n_I$ the dimension of the socle of the $\fh$-module~$S(\fh)/I$.
\end{lemma}

\begin{proof}
For notational convenience we only consider the ideals $I^k$ in this proof. According to Proposition~\ref{projexp}, every projective module in $\cO^k$ is a direct summand of some $\widetilde{M}_{n,k}(\lambda)=\left(M_{n,k}(\lambda)\right)_{\chi_\lambda}$, where $\lambda\in\fh^\ast$ and $n\gg 0$. We prove first that $\pd_{\cO^\infty} \widetilde{M}_{n,k}(\lambda)\le\dim\fh$. 

Take $\nu\in\fh^\ast$ with $\chi_\nu=\chi_\lambda$. For $d>0$, applying \eqref{resultBuch} inside the category $\fg$-mod to the sequence
from  Lemma~\ref{firstly}\eqref{firstly.3} yields the exact sequence
\begin{multline*}
\Ext_{\fg}^{d-1}\big(U(\fg)\fn^n\otimes_{U(\fh)} V_{\lambda,k}, L(\nu)\big)\to 
\Ext^d_{\fg}\big( \widetilde{M}_{n,k}(\lambda),L(\nu)\big)\to\\
\to\Ext^d_{\fg}\big(U(\fg)\otimes_{U(\fh)}V_{\lambda,k},L(\nu)\big)\to 
\Ext_{\fg}^{d}\big(U(\fg)\fn^n\otimes_{U(\fh)} V_{\lambda,k}, L(\nu)\big).
\end{multline*}

We take the projective resolution of $U(\fg)\fn^n\bigotimes_{U(\fh)} V_{\lambda,k}$ 
in $\cW^k$ described in the proof of Lemma~\ref{secondly}. This resolution is given in terms of modules of the form $U(\fg)\otimes_{U(\fh)}V_{\mu,k}$ with all $\mu\not\le\nu$. As for such $\mu$ and for $p>0$ we have
\begin{equation}\label{eq681}
\Ext^p_{\fg}\big(U(\fg)\otimes_{U(\fh)}V_{\mu,k},L(\nu)\big)=
\Ext^p_{\fh}\big(V_{\mu,k},L(\nu)\big)=0,
\end{equation}
our resolution is an acyclic resolution for the functor $\Hom_{\fg}(-,L(\nu))$, which can be used to compute $\Ext^i_{\fg}(-,L(\nu))$. Since \eqref{eq681} is also true for $p=0$, we obtain that 
$$\Ext_{\fg}^{i}\big(U(\fg)\fn^n\otimes_{U(\fh)} V_{\lambda,k}, L(\nu)\big)= 0\quad\mbox{for }i\in\{d-1,d\}.$$

By Frobenius reciprocity we have
\begin{displaymath}
\Ext^d_{\fg}\big(U(\fg)\otimes_{U(\fh)}V_{\lambda,k},L(\nu)\big)\cong
\Ext^d_{\fh}\big(V_{\lambda,k},L(\nu)\big).
\end{displaymath}
By  Theorem~\ref{finalresult} we have
\begin{displaymath}
\Ext^d_{\fg}\big( \widetilde{M}_{n,k}(\lambda),L(\nu)\big)\cong
\Ext^d_{\cO^\infty}\big( \widetilde{M}_{n,k}(\lambda),L(\nu)\big).
\end{displaymath}
The above now implies that, for $d>0$,
\begin{equation}
\label{eqpdM}
\Ext^d_{\cO^\infty}( \widetilde{M}_{n,k}(\lambda),L(\nu))\cong\Ext^d_{\fh}\big(V_{\lambda,k},L(\nu)\big).
\end{equation}
If $\nu$ is not in the orbit of $\lambda$, then all extensions are zero. Since we have $\mathrm{gl.dim}\,(\fh\text{-}\mathrm{mod})=\dim\fh$, we get $\pd_{\cO^\infty}\widetilde{M}_{n,k}\le \dim\fh$. This bound then holds for arbitrary projective modules by Proposition~\ref{projexp}. The fact that this bound is an equality would follow if we could prove equation~\eqref{extrext}, so from now on we focus on that.

We apply equation \eqref{eqpdM} for $d=\dim\fh$, which yields
$$\Ext^{\dim\fh}_{\cO^\infty}( \widetilde{M}_{n,k}(\lambda),L(\nu))\cong\Ext^{\dim\fh}_{\fh}\big(V_{\lambda,k},L(\nu)\big).$$
Lemma \ref{algcohom} then implies
\begin{eqnarray*}
\dim\Ext^{\dim\fh}_{\cO^\infty}( \widetilde{M}_{n,k}(\lambda),L(\nu))&=& \dim\Hom_{\fh}( L(\nu),V_{\lambda,k})\\
&=&n_k\dim\left(L(\nu)^{\lambda}\right),
\end{eqnarray*}
where $n_k$ is the dimension of the socle of $V_{\lambda,k}$. Comparison with Proposition~\ref{projexp} then yields
$$\dim\Ext^{\dim\fh}_{\cO^\infty}( \widetilde{M}_{n,k}(\lambda),L(\mu))=n_k\,\Hom_{\cO^\infty}( \widetilde{M}_{n,k}(\lambda),L(\mu))$$
for all weights $\lambda,\mu$. We claim this implies equation \eqref{extrext}. Indeed, for $\lambda$ dominant, this follows from $\widetilde{M}_{n,k}(\lambda)=P^k(\lambda)$ by Proposition~\ref{projexp}. The full statement then follows by induction along the Bruhat order and Proposition~\ref{projexp}.
\end{proof}

\begin{proof}[Proof of Theorem~\ref{pdthm}]
We prove claim \eqref{pdthm.1} by induction on the projective dimension of $M$ inside $\cO^I$. If the projective dimension is zero, the result follows from Lemma~\ref{helppd}. Now assume $\pd_{\cO^I}M=1$, then there are projective modules $P'$ and $P$ in $\cO^I$ such that we have $P'\hookrightarrow P\tto M$. We denote by $\alpha:P'\hookrightarrow P$ the corresponding morphism between the projective modules, so $M={\rm coker}(\alpha)$. Equation \eqref{pd2} and the result for projective modules implies $\pd_{\cO^\infty}M\le \dim\fh+1$. We can consider the following  part of the long exact sequence~\eqref{resultBuch}
$$\Ext^{\dim\fh}_{\cO^\infty}(P,K)\to \Ext^{\dim\fh}_{\cO^\infty}(P',K)\to\Ext^{\dim\fh+1}_{\cO^\infty}(M,K)\to 0.$$
Therefore we find that $\pd_{\cO^\infty}M=\dim\fh+1$ unless there is a surjection 
$$\Ext^{\dim\fh}_{\cO^\infty}(P,L(\mu))\tto \Ext^{\dim\fh}_{\cO^\infty}(P',L(\mu))$$
for all $\mu\in\fh^\ast$. By Lemma~\ref{helppd} this would imply that $P\cong P'\oplus Q$ for another projective module $Q$ in $\cO^I$. We can choose this direct sum such that under the above surjection $\Ext^{\dim\fh}_{\cO^\infty}(Q,L(\mu))$ gets mapped to zero. Composing $\alpha$ with projection onto the direct summand isomorphic to $P'$ in $P$ then yields an endomorphism ${\widetilde\alpha} \in {\rm End}_{\cO^I}( P')$, such that the induced endomorphism on 
$$\bigoplus_{\mu\in\fh^\ast}\Ext^{\dim\fh}_{\cO^\infty}(P',L(\mu))$$
is an isomorphism. By considering indecomposable direct summands of~$P'$ and Lemma \ref{helppd}, this implies that $\alpha$ always maps an indecomposable projective module in $P'$ to an isomorphic projective module. Assume that $\widetilde\alpha$ is not an isomorphism, then there is a natural number $l$ for which $(\widetilde\alpha)^l\in {\rm End}_{\cO^I}( P')$ annihilates an indecomposable direct summand of $P'$. By Lemma \ref{helppd}, this contradicts the fact that $(\widetilde\alpha)^l$ must still induce an isomorphism on the extensions. Hence $\widetilde\alpha$ is an isomorphism, meaning that the short exact sequence splits and $M\cong Q$ is projective in $\cO^I$, a contradiction with $\pd_{\cO^I}M=1$. Thus the claim is also true for $p=1$.

Now take $p>1$ and assume the result holds for all projective dimensions up to $p-1$. For $M\in\cO^I$, with $\pd_{\cO^I}M=p$, there is a $P$, projective in $\cO^I$, and an $N\in\cO^I$, with $\pd_{\cO^I}N=p-1$, such that $N\hookrightarrow P\tto M$. Formulae~\eqref{pd1} and \eqref{pd2} yield
\begin{eqnarray*}
p+\dim\fh-1&\le&\max\{\dim\fh\,,\,\pd_{\cO^\infty}M\,-1\}\\
\pd_{\cO^\infty} M&\le&\max\{p+\dim\fh\,,\,\dim\fh\},
\end{eqnarray*}
which implies $\pd_{\cO^\infty}M=p+\dim\fh$ if $p>1$ but only $\pd_{\cO^\infty}M\le 1+\dim\fh$ if $p=1$. To include the case $p=1$

Now we focus on claim \eqref{pdthm.2}. Consider $M\in\cO^\infty$, it is a module of finite length with all simple subquotients from $\cO$, therefore it has a weight $\lambda$ such that $M^\lambda\not=0$ and $M$ contains no vectors of higher weights. Similarly to the proof of Lemma~\ref{helppd} (as in formula~\eqref{eqpdM}) we have $$\Ext^{\dim\fh}_{\cO^\infty}(\widetilde{M}_{n,1}(\lambda),M^\star)\cong\Ext^{\dim\fh}_{\fh}\left(\mC_\lambda,M^\star\right)\cong H^{\dim\fh}(\fh,\mC_{-\lambda}\otimes M^\star).$$ 
This is non-zero by Lemma \ref{algcohom}. Equation \eqref{extdual} therefore implies that $M$ has projective dimension at least $\dim\fh$.
\end{proof}

\subsection{Projective dimensions in $\cW^\infty$}\label{secOinf.3}

\begin{theorem}\label{pdthmW}
Let $I'$ be an ideal in $S(\fh)_{(0)}$ of finite codimension.
If $M\in\cW^I$ satisfies $\pd_{\cW^I}M<\infty$, then
$$\pd_{\cW^\infty}M\,\,=\,\,\dim\fh\,+\,\pd_{\cW^I}M.$$
\end{theorem}

\begin{proof}
From Lemma~\ref{lemFrob} and the algebra cohomology of $\fh$ it follows quickly that the projective dimension of projective modules in $\cW^I$ is equal to $\dim\fh$. The result then follows identically as in the 
proof of Theorem~\ref{pdthm}.
\end{proof}

\subsection{Basic classical Lie superalgebras}\label{secOinf.4}

In this subsection we consider basic classical Lie superalgebras, we refer to \cite{Musson} for definitions. We will denote a basic classical Lie superalgebra by $\widetilde{\fg}$ and the underlying Lie algebra of $\widetilde{\fg}$ by $\fg$. An important property of these Lie superalgebras is that the Cartan subalgebra of $\widetilde\fg$ is equal to the one of~$\fg$. Therefore we have natural analogues of the categories introduced in Subsection \ref{secPrel.5} and we denote the corresponding categories by $\widetilde{\cO}^I$, $\widetilde{\cW}^I$ etc.

\begin{theorem}\label{super}
For a basic classical Lie superalgebra $\widetilde{\fg}$  we have:
\begin{enumerate}[$($i$)$]
\item\label{super.1} The BGG category $\widetilde{\cO}$ is extension full in $\widetilde{\cW}$.
\item\label{super.2} The categories $\widetilde{\cO}^\infty$ and $\widetilde{\cW}^\infty$ are extension full in $\widetilde{\fg}$-mod.
\end{enumerate}
\end{theorem}

\begin{proof}
Consider a projective module $\widetilde{P}$ in $\widetilde{\cO}$. 
It is a direct summand of $\mathrm{Ind}_{\fg}^{\widetilde{\fg}}P$ for a projective module $P$ in $\cO$. 
Using the Frobenius reciprocity, we have
$$\Ext^d_{\widetilde{\mathcal{W}}}(\ind^{\widetilde{\fg}}_{\fg} P,M)=
\Ext^d_{{\mathcal{W}}}(P,\res^{\widetilde{\fg}}_{\fg}  M),$$
which is zero for $d>0$ by Corollary \ref{proj}. Claim~\eqref{super.1} now follows from Corollary~\ref{genprop}.

The same reasoning can be used to obtain the extension fullness of $\widetilde{\cO}^I$ into $\widetilde{\cW}^I$, for every ideal $I'$ in $S(\fh)_{(0)}$ of finite codimension. Therefore, Corollary~\ref{inftyabstract} implies that $\widetilde{\cO}^\infty$ is extension full in $\widetilde{\cW}^\infty$. 

Lemma~\ref{projCgh} can be generalized immediately to basic classical Lie superalgebras, since their Cartan subalgebra coincides with the one of the underlying Lie algebra (alternatively, one can use the fact that projective modules in $\widetilde{\cW}^I$ are induced from projective modules in $\cW^I$ and Lemma~\ref{projCgh}). The fact that $\widetilde{\cW}^\infty$ is extension full in $\widetilde{\fg}$-mod then follows from Lemma~\ref{5lem2}.
\end{proof}

\section{Singular blocks in category $\cO$}\label{secsing}

\subsection{Singular blocks in category $\cO$}\label{secsing.1}

Let $\lambda$ be a dominant integral weight for $\mathfrak{g}$ and $W_{\lambda}$ denote the stabilizer of 
$\lambda$ in $W$ with respect to the dot action. Let $w_0$ be the longest element in $W$ and $w^{\lambda}_0$
be the longest element in $W_{\lambda}$. We also denote  by $\mathtt{a}:W\to\mathbb{N}$ 
Lusztig's $\mathtt{a}$-function, see \cite{Lu1,Lu2}.

Consider the corresponding singular block $\mathcal{O}_{\lambda}=\cO_{\chi_\lambda}$.
Then we have the usual exact functors of translation out of and onto the $W_{\lambda}$-wall:
\begin{displaymath}
\theta^{\mathrm{out}}: \mathcal{O}_{\lambda}\to\mathcal{O}_{0}\quad\text{ and }\quad  
\theta^{\mathrm{on}}: \mathcal{O}_{0}\to\mathcal{O}_{\lambda},
\end{displaymath}
see \cite{BG} for details. These functors satisfy 
\begin{equation}\label{eq775}
\theta^{\mathrm{on}}\theta^{\mathrm{out}}\cong
\mathrm{Id}_{\mathcal{O}_{\lambda}}^{\oplus|W_{\lambda}|}.
\end{equation}
Furthermore, the functor 
$\theta^{\mathrm{out}}\theta^{\mathrm{on}}$ is the unique indecomposable projective endofunctor of 
$\mathcal{O}_{0}$ sending $M(0)$ to the projective cover of $L(w^{\lambda}_0\cdot 0)$. This functor
is usually denoted $\theta_{w^{\lambda}_0}$. The main result of this section is the following observation.

\begin{theorem}\label{thm777}
\begin{enumerate}[$($i$)$]
\item\label{thm777.1} The projective dimension of the simple Verma mo\-du\-le in $\mathcal{O}_{\lambda}$
equals $\mathtt{a}(w_0w^{\lambda}_0)$. 
\item\label{thm777.2} We have $\mathrm{gl.dim}\,\mathcal{O}_{\lambda}=2\mathtt{a}(w_0w^{\lambda}_0)$. 
\item\label{thm777.3} The projective dimension of the dominant simple module in $\mathcal{O}_{\lambda}$
equals $2\mathtt{a}(w_0w^{\lambda}_0)$. 
\end{enumerate}
\end{theorem}

\begin{proof}
Let $L$ be the simple Verma module in $\mathcal{O}_{\lambda}$. Then $L\cong \theta^{\mathrm{on}} L(w_0\cdot 0)$
and hence
\begin{displaymath}
\theta^{\mathrm{out}}L\cong  \theta^{\mathrm{out}}\theta^{\mathrm{on}} L(w_0\cdot 0)=
\theta_{w^{\lambda}_0}L(w_0\cdot 0)
\end{displaymath}
is the indecomposable tilting module $T( w_0w^{\lambda}_0\cdot 0)$ in $\mathcal{O}_0$ 
with highest weight $ w_0w^{\lambda}_0\cdot 0$. By \cite[Theorem~17]{MR2602033}, the projective 
dimension of $T( w_0w^{\lambda}_0\cdot 0)$ equals $\mathtt{a}( w_0w^{\lambda}_0)$. On the other hand,
from \eqref{eq775} it follows that $\theta^{\mathrm{on}}T( w_0w^{\lambda}_0\cdot 0)$ is a direct sum of copies of $L$.
As projective functors are exact and send projective modules to projective modules, it follows that
\begin{eqnarray*}
\pd_{\cO_\lambda} L=\pd_{\cO_\lambda} L^{\oplus |W_\lambda|}= \pd_{\cO_\lambda}\theta^{on} T( w_0w^{\lambda}_0\cdot 0)&\le& \pd_{\cO_0}T( w_0w^{\lambda}_0\cdot 0); \\
\pd_{\cO_0}T( w_0w^{\lambda}_0\cdot 0) =\pd_{\cO_0}\theta^{out}L&\le &\pd_{\cO_\lambda}L.\end{eqnarray*}
Hence the projective dimensions of $L$ and $T( w_0w^{\lambda}_0\cdot 0)$ coincide, proving claim~\eqref{thm777.1}.

The parabolic-singular Koszul duality from \cite{BGS} asserts that the Koszul dual of $\mathcal{O}_{\lambda}$
is the parabolic subcategory $\mathcal{O}^{W_{\lambda}}_0$ of $\mathcal{O}_0$ associated to 
$W_{\lambda}$. We use the normalization of Koszul duality which maps simple objects to indecomposable 
injective objects. By the graded length of a module we mean the number of non-zero graded components with 
respect to Koszul grading. Then Koszul duality maps a simple module of projective dimension $p$ to an 
indecomposable injective  module of graded length  $p+1$ and reverses the quasi-hereditary order.
Therefore Koszul duality maps $L$ to the dominant costandard module in $\mathcal{O}^{W_{\lambda}}_0$.
We denote the latter module by $M$, which thus has graded length $\mathtt{a}(w_0w^{\lambda}_0)+1$, by claim~\eqref{thm777.1}. The injective envelope $I$ of the dual module $M^{\star}$ (the dominant standard module) 
is known to be projective-injective and hence tilting, see e.g. \cite[Section~3]{MR2602033}. This is the
only tilting module which contains the dominant simple as a subquotient. Therefore $I$ is the tilting module
associated to the standard module $M^{\star}$, i.e. we have a (unique up to a nonzero scalar) 
injection $M^{\star}\hookrightarrow I$ and a (unique up to a nonzero scalar)
surjection $I\tto M$ and the image of the composition of these two maps coincides with the simple socle of $M$. As
the socles of $I$ and $M^{\star}$ agree and, at the same time, the heads of $I$ and $M$ agree, it
follows that 
\begin{displaymath}
\mathrm{graded\,\,\, length}(I)=
\mathrm{graded \,\,\,length}(M)+\mathrm{graded\,\,\, length}(M^{\star})-1.
\end{displaymath}
Clearly, the graded lengths of $M$ and $M^\star$ coincide. In \cite[Section~3]{MR2602033} it is shown that all
projective-injective modules in $\mathcal{O}^{W_{\lambda}}_0$ have the same graded length and that each projective
module is a submodule of a projective-injective module. This implies that the maximal graded length of an
indecomposable injective module in $\mathcal{O}^{W_{\lambda}}_0$ is $2\mathtt{a}( w_0w^{\lambda}_0)+1$
which implies claim~\eqref{thm777.2}.

Claim~\eqref{thm777.3} follows from the fact that Koszul duality maps the dominant simple module in $\mathcal{O}_{\lambda}$ to the antidominant injective in the parabolic category and the latter module
is automatically projective and hence has maximal graded length, as mentioned in the previous paragraph.
This completes the proof.
\end{proof}

\subsection{The $\mathfrak{sl}_3$-examples}\label{secsing.2}

As described in \cite[Section 5.2.1]{St}, any non-trivial singular integral block of category~$\mathcal{O}$
for $\mathfrak{sl}_3$ is equivalent to the category of modules over the following quiver with relations:
\begin{equation}\label{eq734}
\xymatrix{ 
1\ar@/^/[rr]^{a}&&2\ar@/^/[rr]^{c}\ar@/^/[ll]^{b}&&3\ar@/^/[ll]^{d}
}\qquad cd=0,\,\, ab=dc.
\end{equation}
Let $L_i$ for $i=1,2,3$ be simple modules corresponding to vertices in this quiver and $P_i$ be their
projective covers. Then $P_3$, $P_2$ and $P_1$ have the following Loewy structure,
respectively:
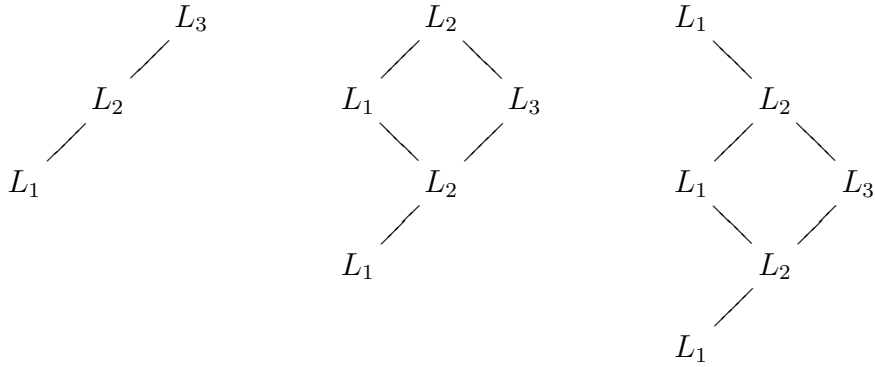
\begin{displaymath}
\xymatrix@!=0.6pc{
&&L_3\ar@{-}[dl]&&&L_2\ar@{-}[dl]\ar@{-}[dr]&&&L_1\ar@{-}[dr]&&\\
&L_2\ar@{-}[dl]&&&L_1\ar@{-}[dr]&&L_3\ar@{-}[dl]&&&L_2\ar@{-}[dl]\ar@{-}[dr]&\\
L_1&&&&&L_2\ar@{-}[dl]&&&L_1\ar@{-}[dr]&&L_3\ar@{-}[dl]\\
&&&&L_1&&&&&L_2\ar@{-}[dl]&\\
&&&&&&&&L_1&&
}  
\end{displaymath}
A direct computation thus implies
\begin{displaymath}
\mathrm{pd}\,L_1=1,\quad \mathrm{pd}\,L_2=\mathrm{pd}\,L_3=2.
\end{displaymath}
In particular, the global dimension of this module category equals $2$.
All this fully agrees with Theorem~\ref{thm777} and with \cite{MR2065514}. 

Further, it is straightforward to check that the minimal projective resolution of $L_3$ has the following form:
\begin{displaymath}
0\to P_3\to P_2\to P_3\to L_3\to 0.
\end{displaymath}
This implies that we have 
\begin{equation}\label{eq736}
\mathrm{Ext}^2(L_3,L_3)\neq 0
\end{equation}
in this module category.

From the calculation of projective dimensions and the quiver it follows that the Serre subcategory generated by the simple module $L_3$ is an initial segment. This follows from the calculation of the projective dimensions and
\begin{displaymath}
 \mathrm{Ext}^1(L_1,L_3)=\mathrm{Ext}^1(L_3,L_1)=0,
\end{displaymath}
because the vertices $1$ and $3$ in the quiver \eqref{eq734} are not connected. Since $\Ext^1(L_3,L_3)=0$, 
this initial segment is semi-simple.
However, it is not extension full by \eqref{eq736}.

Hence we find that there is a singular integral block in category $\cO$ which is not Guichardet.

\subsection{Open questions for singular blocks}\label{secsing.3}

It would be interesting to generalize the explicit description of homological invariants for 
structural modules in the block $\mathcal{O}_0$ described in \cite{MR2366357,MR2602033}, including 
projective dimension of simple, standard, indecomposable tilting and indecomposable injective modules,
to the singular case. Theorem~\ref{thm777} already makes some steps in this direction. An observation from 
Subsection \ref{secsing.2} is that, contrary to regular blocks,
the projective dimension of simple modules in singular blocks is no longer strictly monotone along the
Bruhat order. This is precisely the origin of the failure of the block to be Guichardet. In this example
these projective dimensions are still weakly monotone, which implies that the block will satisfy
a weaker notion of Guichardet category, studied in \cite{Fu1}.

The projective dimension of structural modules in singular blocks (and parabolic category $\cO$) 
will be studied in more detail in \cite{SHPO4}. 
One of the consequences is that, in general, also the weaker version of the Alexandru conjecture, as studied in
\cite{Fu1}, is not true for singular blocks in category $\cO$. In fact, it turns out that 
the projective dimensions of simple modules in singular blocks are even not weakly monotone along
the Bruhat order.

Using the same arguments as in the proof of Theorem~\ref{thm777} one shows that computation of 
projective dimension of simple modules in $\mathcal{O}_{\lambda}$ is equivalent to
computation of projective dimension in $\mathcal{O}_0$ of the modules
$\theta_{w^{\lambda}_0}L(w\cdot 0)$ where $w$ is a longest coset representative in 
$W/W_{\lambda}$. This is a special case of the general open question in \cite[Problem~24]{MR2602033}.

\section{Regular blocks of  $\cO$ and $\cO^\infty$ are Guichardet}\label{secwac}

Our main result in this section is the following.

\begin{theorem}\label{WAC}
Let $\fg$ be a semisimple complex Lie algebra.
\begin{enumerate}[$($i$)$]
\item\label{WAC.1} 
For $\chi$ a regular central character, both categories $\cO_\chi$ and $\cO^\infty_\chi$ are Guichardet.
\item\label{WAC.2} 
For $\theta$ a singular central character, the categories $\cO_\theta$ is not always Guichardet.
\end{enumerate}
\end{theorem}

The first step in proving Theorem \ref{WAC} is determining initial segments in the categories $\cO$ and $\cO^\infty$.

A {\em coideal} $\Gamma_W$ in the Weyl group $W$, with respect to the Bruhat order $\geq$, is a subset of $W$ such that $w'\ge w$ and $w\in \Gamma_W$ imply $w' \in \Gamma_W$. We use the same conventions for the Bruhat order as in  \cite[Section~0.4]{MR2428237}.

An {\em ideal} $\Gamma$ in $\fh^\ast$ is a subset such that $\lambda'\le\lambda$ and $\lambda\in\Gamma$ imply $\lambda'\in\Gamma$. For an integral dominant $\lambda\in\fh^\ast$, we have $w\cdot\lambda\ge w'\cdot\lambda$ if and only if $w\le w'$, so there is a one to one correspondence between coideals in $W$ and ideals in $\fh^\ast$ contained in the orbit of a fixed integral regular weight.

\begin{lemma}\label{initseg}
\begin{enumerate}[$($i$)$]
\item\label{initseg.1} 
Consider a regular central character $\chi$. The initial segments in $\cO_\chi$ are the full Serre subcategories generated by a set 
of modules of the form $\{L(\lambda)|\lambda\in \Gamma\}$, for $\Gamma$ some ideal in $\{\lambda\in\fh^\ast\,|\,\chi_\lambda=\chi\}$.
\item\label{initseg.2} 
The initial segments in $\cO^\infty$ are the Serre subcategories generated by the initial segments in $\cO$.
\end{enumerate}
\end{lemma}

\begin{proof}
We start with the principal block $\cO_0$ and show that the initial segments are the 
full Serre subcategories generated by a set 
of modules of the form $\{L(w\cdot 0)|w\in \Gamma_W\}$, for $\Gamma_W$ some coideal in $W$.

By \cite[Proposition~6]{MR2366357} we have $\pd_{\cO} L(w)=2l(w_0)-l(w)$.
The $\Ext^1$-quiver of $\cO_0$ is known as a consequence of the Kazhdan-Lusztig conjecture, see e.g.  \cite[Section~7]{twisting}. In particular, we have
\begin{itemize}
\item if $w'\ge w$ and $l(w')=l(w)+1$, then $\Ext^1_{\cO}(L(w),L(w'))\not=0$;
\item if $l(w')=l(w)+1$ and $w$ and $w'$ are not comparable, then $\Ext^1_{\cO}(L(w),L(w'))=0$.
\end{itemize}
The first property implies that every initial segment in $\cO_0$ corresponds to a coideal in $W$. The second property implies that every coideal in $W$ corresponds to an initial segment. This proves claim $(i)$ for $\cO_0$.

Next, we consider an indecomposable block inside $\cO_{\chi}$, for $\chi$ a non-integral regular central character. By \cite{SoergelD}, such a block is equivalent to some regular integral indecomposable block in category $\cO$ (possibly for a different Lie algebra), where the equivalence preserves the highest weight structure. Since an equivalence of categories maps initial segments to initial segments, claim~\eqref{initseg.1} follows.

Now we turn to $\cO^\infty$ for arbitrary central characters. We argue that the $\Ext^1$-quiver of 
$\cO^\infty$ is the same one as for $\cO$, up to loops (that is self-extensions of simple modules). 
Imagine there is a module $M\not\in\cO$ satisfying
$$L(\lambda)\hookrightarrow M\tto L(\lambda')$$
for $\lambda,\lambda'\in \fh^\ast$. This module is clearly in $\cO^2$. If $\lambda'\not\le \lambda$, then $M$ is a quotient of $M(\lambda')$ by the universal property of Verma modules. If $\lambda'< \lambda$ we can use the duality $\star$ on $\cO^2$, which preserves $\cO$, to return to the previous situation. Therefore $\lambda=\lambda'$.

Going from $\cO$ to $\cO^\infty$, we therefore have that the extension quivers coincide up to self-extensions and the projective dimensions of simple modules coincide up to a shift by $\dim\fh$, see Theorem~\ref{pdthm}. Therefore claim~\eqref{initseg.2} follows.
\end{proof}

Lemma \ref{initseg} allows us to apply a result on stratified algebras by Cline, Parshall and Scott in
\cite{CPS}, or the special case of quasi-hereditary algebras in \cite{CPS1}.

\begin{proof}[Proof of Theorem \ref{WAC}]
Lemma \ref{initseg}\eqref{initseg.1} implies that, in every regular block, the initial segments correspond to ideals in the poset of weights. The property for $\cO_{\chi}$ with $\chi$ regular therefore follows from \cite[Theorem~3.9(i)]{CPS1}.

Indecomposable blocks of category $\cO^I$ are stratified (in the sense of \cite[Definition 2.2.1]{CPS})
with respect to the order $\le$ on the weights, see  \cite[Lemma~8]{Soergel}
and \cite[Theorem~7]{Soergel} or \cite[Corollary~9(a)]{KKM}. Take $\Gamma$ an ideal in the poset (of a regular block) and let $\cO^I_{\Gamma}$ be the Serre subcategory of $\cO^I$ generated by $\{L(\lambda)\,|\,\lambda\in \Gamma\}$. Then
$$\Ext^j_{\cO^I_{\Gamma}}(M,N)\to \Ext^j_{\cO^I}(M,N)$$
is an isomorphism for any $M,N\in\cO^I_\Gamma$, see e.g. \cite[Equation 2.1.2.1]{CPS}. 
Corollary~\ref{inftyabstract} then implies that 
$$\Ext^d_{\cO^\infty_{\Gamma}}(M,N)\to \Ext^d_{\cO^\infty}(M,N)$$
is an isomorphism for any $M,N\in\cO^\infty_{\Gamma}$ and $d\ge 0$, so $\cO^\infty_{\chi}$ is 
Guichardet by Lemma~\ref{initseg}\eqref{initseg.2}, which proves claim~\eqref{WAC.1}.

Claim~\eqref{WAC.2} follows immediately from the singular example for $\mathfrak{sl}(3)$ described in Subsection~\ref{secsing.2}.
\end{proof}

\section{Harish-Chandra bimodules}\label{sechc}

Let $\chi$ be a regular central character and $\theta$ be a central character in the same weight lattice as $\chi$.
The equivalences of categories in \cite[Theorem~5.9]{BG} and
\cite[Theorem~1]{SoergelF} imply that for $k\geq 1$ the category $\cO^k_{\theta}$ is equivalent to both ${}_{\chi}^k\cH^\infty_{\theta}$ and ${}_{\,\,\theta}^{\infty}\cH^k_{\chi}$ and that $\cO^\infty_{\theta}$ is equivalent to both ${}_{\,\,\theta}^\infty\cH^\infty_{\chi}$ and ${}_{\,\,\chi}^\infty\cH^\infty_{\theta}$. 
As a consequence, the claims of the following  result follow from 
Corollary~\ref{gldiminfty}, Theorem~\ref{pdthm} and Theorem~\ref{WAC}, respectively.

\begin{theorem}\label{thmHC}
Let $\chi$ be a regular central character and $\theta$ be a central character in the same weight lattice as $\chi$.
\begin{enumerate}[$($i$)$]
\item\label{thmHC.1} 
The global dimension of the category ${}_{\,\,\chi}^\infty\cH^\infty_{\theta}$ is finite. If $\chi$ is integral, then the global dimension of ${}_{\,\,\chi}^\infty\cH^\infty_{\chi}$ is $\dim\fg$.
\item\label{thmHC.2} 
Consider $M\in{}_{\chi}^k\cH^\infty_{\theta}$ with $\pd_{{}_{\chi}^k\cH^\infty_{\theta}}M<\infty$, then
$$\pd_{{}_{\,\,\chi}^\infty\cH^\infty_{\theta}}M\,\,=\,\,\dim\fh\,+\,\pd_{{}_{\chi}^k\cH^\infty_{\theta}}M.$$
\item\label{thmHC.3} 
If $\theta$ is also regular, the categories ${}_{\,\,\theta}^\infty\cH^\infty_{\chi}$, ${}_{\,\,\chi}^\infty\cH^\infty_{\theta}$,  ${}_{\chi}^1\cH^\infty_{\theta}$  and ${}_{\,\,\theta}^{\infty}\cH_{\chi}^1$ are Guichardet.
\end{enumerate}
\end{theorem}

We conclude with the result that, in spite of Theorem \ref{finalresult}, the category 
${}_{\,\,\chi}^\infty\cH^\infty_{\chi}$ is not extension full in the category of bimodules.

\begin{proposition}\label{proplast}
The category ${}_{\,\,\chi}^\infty\cH^\infty_{\chi}$ for $\chi$ a regular integral central character is not extension full in the category of $\fg$-bimodules.
\end{proposition}

\begin{proof}
Without loss of generality we may assume that  $\chi$ is the central character of the trivial $\fg$-module $\mC$. As noted in Theorem \ref{thmHC}, the global dimension of the category ${}_{\,\,\chi}^\infty\cH^\infty_{\chi}$ is $\dim\fg$. The trivial $\fg$-bimodule $\mC$ is an object of ${}_{\,\,\chi}^\infty\cH^\infty_{\chi}$. Identifying $\fg$-bimodules with $\fg\oplus\fg$-modules, we find that
$$\Ext^{2\dim\fg}_{\fg\text{-}\mathrm{mod}\text{-}\fg}(\mC,\mC)\cong H^{2\dim\fg}(\fg\oplus\fg,\mC)\not=0,$$
by Lemma \ref{algcohom}. This implies that $$0= \Ext^{2\dim\fg}_{{}_{\,\,\chi}^\infty\cH^\infty_{\chi}}(\mC,\mC)\,\,\,\to\,\,\,\Ext^{2\dim\fg}_{\fg\text{-}\mathrm{mod}\text{-}\fg}(\mC,\mC)$$ can not be an isomorphism.
\end{proof}
\vspace{1cm}

\noindent
{\bf Acknowledgment.}
KC is a Postdoctoral Fellow of the Research Foundation - Flanders (FWO).
VM is partially supported by the Swedish Research Council,
Knut and Alice Wallenbergs Stiftelse and the Royal Swedish Academy of Sciences.
We thank Vera Serganova and Catharina Stroppel for useful discussions.
We thank the referee for helpful comments.

\vspace{5mm}

\noindent
KC: Department of Mathematical Analysis, Ghent University, Krijgs-laan 281, 9000 Gent, Belgium;
E-mail: {\tt Coulembier@cage.ugent.be} 
\vspace{2mm}

\noindent
VM: Department of Mathematics, University of Uppsala, Box 480, SE-75106, Uppsala, Sweden;
E-mail: {\tt  mazor@math.uu.se}
\date{}


\vspace{1cm}


\begin{thebibliography}{99999} 
\bibitem[AS]{twisting}
H.~Andersen, C.~Stroppel.
Twisting functors on $\mathcal{O}$.
Represent. Theory {\bf 7} (2003), 681--699
\bibitem[BGS]{BGS}
A.~Beilinson, V.~Ginzburg, W.~Soergel. Koszul duality patterns in representation theory. 
J. Amer. Math. Soc. {\bf 9} (1996), no. 2, 473--527. 
\bibitem[Bu]{Buchsbaum}
D.~Buchsbaum.
A note on homology in categories. 
Ann. of Math. (2) {\bf 69} (1959), 66--74. 
\bibitem[BG]{BG}
I.~Bern\v{s}tein, S.~Gel'fand.
Tensor products of finite- and infinite-dimensional 
representations of semisimple Lie algebras. 
Compositio Math. {\bf 41} (1980), no. 2, 245--285.
\bibitem[BGG]{BGG}
I.~Bern\v{s}tein, I.~Gel'fand, S.~Gel'fand.
A certain category of $\mathfrak{g}$-modules.
Funkcional. Anal. i Prilo\v{z}en. {\bf 10} (1976), no. 2, 1--8. 
\bibitem[CPS1]{CPS1}
E.~Cline, B.~Parshall and L.~Scott.
Finite-dimensional algebras and highest weight categories. 
J. Reine Angew. Math. {\bf 391} (1988), 85--99. 
\bibitem[CPS2]{CPS}
E.~Cline, B.~Parshall, L.~Scott.
Stratifying endomorphism algebras.
Mem. Amer. Math. Soc. {\bf 124} (1996), no. 591. 
\bibitem[CM]{SHPO4}
K.~Coulembier, V.~Mazorchuk.
Some homological properties of category~$\cO$. IV.
In preparation.

\bibitem[De]{De}
P.~Delorme. Extensions in the Bernsteĭn-Gelfand-Gelfand category $\mathcal{O}$. 
Funct Anal. Appl. {\bf 14} (1980), no. 3, 77--78.
\bibitem[Fu1]{Fu1}
A.~Fuser. Autour de la conjecture d'Alexandru. PhD Thesis, Nancy, 1997
\bibitem[Fu2]{Fu2}
A.~Fuser. Around the Alexandru Conjecture. Prepublication de l'Institut Elie Cartan, Nancy, 1997
\bibitem[Fu3]{Fu3}
A.~Fuser. The Alexandru Conjectures. Prepublication de l'Institut Elie Cartan, Nancy, 1997
\bibitem[Fu4]{Fu4}
A.~Fuser. The Alexandru conjecture for Harish-Chandra modules. The rank one case. Manuscript
\bibitem[Fu5]{Fu5}
A.~Fuser. From Vogan conjecture to Alexandru conjectures. Manuscript
\bibitem[Ga1]{Ga1}
P.-Y.~Gaillard. Introduction to the Alexandru Conjecture. Preprint arXiv:math/0003069.
\bibitem[Ga2]{Ga2}
P.-Y.~Gaillard. Statement of the Alexandru Conjecture. Preprint arXiv:math/0003070.
\bibitem[He]{He}
R.~Hermann. Monoidal categories and the Gerstenhaber bracket in Hochschild cohomology.  Preprint arXiv:1403.3597. To appear in "Memoirs of the American Mathematical Society".
\bibitem[Hu]{MR2428237}
J.~Humphreys.
\newblock{Representations of semisimple Lie algebras in the BGG category O.}
\newblock{Graduate Studies in Mathematics, {\bf 94}. American Mathematical Society, Providence, RI (2008).}
\bibitem[KKM]{KKM}
O.~Khomenko, S.~K\"onig, V.~Mazorchuk.
Finitistic dimension and tilting modules for stratified algebras.
J. Algebra {\bf 286} (2005), no. 2, 456--475. 
\bibitem[Lu1]{Lu1}
{G.~Lusztig}. Cells in affine Weyl groups. Algebraic groups and 
related topics (Kyoto/Nagoya, 1983), 255--287, Adv. Stud. Pure 
Math., {\bf 6}, North-Holland, Amsterdam, 1985. 
\bibitem[Lu2]{Lu2}
{G.~Lusztig}. Cells in affine Weyl groups. II.
J. Algebra {\bf 109} (1987), no. {2}, 536-548.
\bibitem[Mc]{Maclane}
S.~Mac Lane.
Homology.  Classics in Mathematics. Springer-Verlag, Berlin, 1995.
\bibitem[Ma1]{MR2366357}
V.~Mazorchuk.
Some homological properties of the category O. 
Pacific J. Math. {\bf 232} (2007), no. 2, 313--341. 
\bibitem[Ma2]{MR2602033}
V.~Mazorchuk.
Some homological properties of the category O. II.
Represent. Theory {\bf 14} (2010), 249--263.
\bibitem[Ma3]{MazNN}
V.~Mazorchuk.
Lectures on $\mathfrak{sl}_2(\mathbb{C})$-modules. 
Imperial College Press, London, 2010. x+263 pp.
\bibitem[MaO]{MR2065514} 
V.~Mazorchuk, S.~Ovsienko.
Finitistic dimension of properly stratified algebras.
Adv. Math. {\bf 186} (2004), no. 1, 251--265.
\bibitem[Mu]{Musson}
I.~Musson.
Lie superalgebras and enveloping algebras. 
Graduate Studies in Mathematics, 131. American Mathematical Society, Providence, RI, 2012.
\bibitem[Ps]{Ps}
C.~Psaroudakis.
Homological theory of recollements of abelian categories.
J. Algebra {\bf398} (2014), 63--110. 
\bibitem[So1]{SoergelF}
W.~Soergel.
\'Equivalences de certaines cat\'egories de g-modules.
C. R. Acad. Sci. Paris S\'er. I Math. {\bf 303} (1986), no. 15, 725--728. 
\bibitem[So2]{SoergelD}
W.~Soergel.
Kategorie O, perverse Garben und Moduln \"uber den Koinvarianten zur Weylgruppe.
J. Amer. Math. Soc. {\bf 3} (1990), no. 2, 421--445. 
\bibitem[So3]{Soergel}
W.~Soergel.
The combinatorics of Harish-Chandra bimodules. 
J. Reine Angew. Math. {\bf 429} (1992), 49--74.
\bibitem[St]{St}
C.~Stroppel.
Category O: quivers and endomorphism rings of projectives.
Represent. Theory {\bf 7} (2003), 322--345. 
\bibitem[Ve]{Ve}
J.L.~Verdier.
Des cat\'egories d\'eriv\'ees des cat\'egories ab\'eliennes.
Ast\'erisque No. {\bf239}, 1996. 
\bibitem[We]{Weibel} 
C.~Weibel.
An introduction to homological algebra.
Cambridge Studies in Advanced Mathematics, {\bf 38}. Cambridge University Press, Cambridge, 1994.
\end{thebibliography}
\end{document}